\documentclass{amsart}
\usepackage[utf8]{inputenc}
\usepackage[T1]{fontenc}
\usepackage[english]{babel}
\usepackage{amsmath}
\usepackage{amsfonts}
\usepackage{amssymb}
\usepackage{graphicx}

\usepackage{xcolor}

\usepackage{amsthm}
\newtheorem{theorem}{Theorem}

\theoremstyle{definition}

\theoremstyle{definition}
\newtheorem{proposition}{Proposition}

\theoremstyle{definition}

\DeclareMathOperator{\Id}{Id}
\DeclareMathOperator{\Span}{Span}
\DeclareMathOperator{\sign}{sign}

\title[$*$-graded polynomial identities for $M_{1,1}(E)$, $UT_{1,1}(E)$ and $UT_{3}(E)$]{$\mathbb{Z}_2$-graded $*$-polynomial identities and cocharacteres  for $M_{1,1}(E)$, $UT_{1,1}(E)$ and $UT_{(0,1,0)}(E)$}
\author{Jonatan Andres Gomez Parada}
\address{Department of Mathematics \\ IMECC, UNICAMP \\ 
Sérgio Buarque de Holanda, 651, 13083-859 Campinas, SP, Brazil}
\email{j211980@dac.unicamp.br}
\thanks{This study was financed in part by the Coordena\c c\~ao de Aperfei\c coamento de Pessoal de N\'{\i}vel Superior - Brasil (CAPES) -
Finance Code 001.}

\date{}
\begin{document}

\begin{abstract}
   Let $K$ be a field of characteristic 0, and let $E$ be the infinite-dimensional Grassmann algebra over $K$. We consider $E$ as a $\mathbb{Z}_2$-graded algebra, where the grading is given by the vector subspaces $E_0$ and $E_1$, consisting of monomials of even and odd lengths, respectively. Thus, if $A = A_0 \oplus A_1$ is an associative $\mathbb{Z}_2$-graded algebra, we can consider the $\mathbb{Z}_2$-graded algebra  $A \hat{\otimes} E = (A_0 \otimes E_0) \oplus (A_1 \otimes E_1)$. In case both $E$ and $A$ are endowed with superinvolutions, we can define a $\mathbb{Z}_2$-graded involution on $A \hat{\otimes} E$ induced by the respective superinvolutions. In this paper, we consider the $\mathbb{Z}_2$-graded  matrix algebras $M_{1,1}(K)$, $UT_{1,1}(K)$, and $UT_{(0,1,0)}(K)$  endowed with superinvolutions. We shall provide a description of the polynomial identities and the cocharacter sequences of $M_{1,1}(K)\hat{\otimes} E$, $UT_{1,1}(K)\hat{\otimes} E$, and $UT_{(0,1,0)}(K)\hat{\otimes} E$, considering these resulting algebras as $\mathbb{Z}_2$-graded algebras with graded involution.    
\end{abstract}


\maketitle

\section*{Introduction}
Matrix algebras with entries in Grassmann algebras have been the subject of various studies. Around 1984--86, A. Kemer developed a sophisticated theory of the ideals of identities in the free associative algebra, over a field of characteristic 0, see his monograph \cite{kemer1991ideals}. According to that theory, the ideals of identities of the matrix algebras, the matrix algebras with entries in the Grassmann algebra, and certain subalgebras of the latter algebra, form, in  a sense, the building blocks for all T-ideals. We define these algebras in the following paragraphs. Of particular interest is determining a basis of their polynomial identities as well as their corresponding cocharacters. Let $K$ be a field of characteristic different from 2, and consider  the infinite dimensional Grassmann algebra $E  = \langle 1, e_1, e_2, \dots \mid  e_i e_j = -e_j e_i \rangle $ over $K$ as  a superalgebra, that is, a $\mathbb{Z}_2$-graded algebra. It is well known that a basis of the vector space $E$ is given by the monomials $e_{i_1}e_{i_2}\cdots e_{i_k}$, $i_1<i_2<\cdots<i_k$, $k\ge 0$. Here we assume that $k=0$ corresponds to the unit element 1. The grading is given by the subspaces $E_0$ and $E_1$, of monomials of even and of odd length respectively.  Given a superalgebra $A = A_0 \oplus A_1$, one considers the \textsl{Grassmann envelope} of $A$ defined by $G(A) = (A_0\otimes E_0) \oplus (A_1\otimes E_1)$. In general, given $A$ and $B$ two $\mathbb{Z}_2$-graded algebras we can define the superalgebra $A \hat{\otimes} B $ given by   the {\sl super tensor product} $A \hat{\otimes} B := (A_0\otimes B_0) \oplus (A_1\otimes B_1)$. Tensor products involving algebras and superalgebras have been a subject of interest in the theory of PI-algebras. Another type of these tensor products (its twisted version) was considered by Regev and Seeman in \cite{regev2005z2}.

Let  $M_n(K)$ be the matrix algebra of order $n\times n$ over $K$ and  $M_n(E)$ the matrix algebra of order $n\times n$ over the Grassmann algebra $E$. According to the theory developed by A. Kemer, it can be deduced that if ${\rm char} K = 0$ the only non-trivial T-prime T-ideals in $K\langle X \rangle$ are $T(M_n(K))$, $T(M_n(E))$, for $n \geq 1$; and $T(M_{a,b}(E))$  for  $a \geq b \geq 1$, where $M_{a,b}(E)$ means the super tensor product $M_{a,b}(K) \hat{\otimes} E $, and $M_{a,b}(K)$  is the algebra $M_{a+b}(K)$ with $\mathbb{Z}_2$-grading given by  $M_{a+b}(K)_0 = \begin{pmatrix} 
B_1 & 0 \\
0 & B_4     
\end{pmatrix}$ and 
$M_{a+b}(K)_1 =\begin{pmatrix} 
0 & B_2 \\
B_3 & 0     
\end{pmatrix}$, with $B_1 \in M_a(K)$, $B_4 \in M_b(K)$, and $B_2$ and $B_3$ are matrices $a\times b$ and $b\times a$ respectively.    It follows from Kemer's  Tensor Product Theorem that if ${\rm char} K = 0$, then $M_{1,1}(E)$ and $E\otimes E$ share the same polynomial identities. The theorem further yields that the tensor product of two T-prime algebras is PI-equivalent to a T-prime algebra as well. However, this does not hold in the case of positive characteristic, as shown in \cite{azevedo2004tensor}.

  In \cite{di2011polynomial}, Di Vincenzo and Koshlukov studied the  identities of the algebra $M_{1,1}(E)$ as an algebra with  involution, while da Silva in \cite{da2012Z2} considered the $\mathbb{Z}_2$-graded identities of $UT_2(K) \hat{\otimes} E$. A generalization to $UT_{k,l}(K) \hat{\otimes} E$ was presented by Di Vincenzo and da Silva in \cite{di2014z2}. Centrone and da Silva studied in \cite{centrone2015z2} the case of $\mathbb{Z}_2$-graded identities of $UT_2(E)$ in characteristic different from $2$. Also, Centrone in \cite{centrone2011ordinary} considered ordinary and $\mathbb{Z}_2$-graded cocharacters of $UT_2(E)$.

A natural further step is to consider involutions and graded involutions on matrix algebras with entries in Grassmann algebras. An {\sl involution} on an algebra $A$ is an antiautomorphism of order two, that is, a linear map $* \colon A \to  A$ satisfying $(ab)^{*} = b^{*}a^{*}$ and $(a^{*})^{*} = a$, for every $a$, $b \in  A$. In the classification of involutions, such ones are called involutions of the first kind. A $G$-graded algebra $A = \oplus_{g\in G} A_g$ with involution $*$ is called a {\sl graded involution algebra} if $(A_g)^* = A_g$ for every $g \in G$. In this case, we say that $*$ is a {\sl graded involution} on $A$. If $A$ is a graded involution algebra, we say it is a $(G,*)$-algebra. 

Given a superalgebra $A = A_0 \oplus A_1$, a {\sl superinvolution} $*$ on $A$ is a graded linear map of order 2 such that $ (ab)^* = (-1)^{|a||b|}b^* a^* $,  for any homogeneous elements $a$, $b \in  A_0 \cup A_1$. Here $|x|$ denotes the homogeneous degree of $x \in  A_0 \cup A_1$. The $\mathbb{Z}_2$-graded involutions and the superinvolutions of the upper triangular matrix algebra $UT_n(K)$ were described by Ioppolo and Martino in \cite{ioppolo2018superinvolutions}, and the corresponding theorem about the superinvolutions of $M_{1,1}(K)$ was obtained by Gomez and Shestakov in \cite{gomez1998lie}. In the case of the Grassmann algebra $E$, we can define the superinvolutions  $i_E$ and $-i_E$ induced by $id_E$ and $-id_E$, where $id_E$ is the identity map on the generators $e_i$ of $E$. That is the superinvolution $i_E$ is equal to the identity map on $E$, and  for  $-i_E$ we have $ -i_E(a) = a$, $-i_E(c) = -c$, for $a\in E_0$, $c \in E_1$, see for example \cite{giambruno2019superalgebras}.  Note that if $\circledast$ and $\diamond$ is a pair of superinvolutions defined on the superalgebras $A$ and $B$ respectively, then the map $*$ defined on  $A \hat{\otimes} B = (A_0 \otimes  B_0) \oplus (A_1 \otimes B_1)$ by putting $(a \otimes b)^* = a^{\circledast} \otimes b^{\diamond}$ is an involution on $A \hat{\otimes} B$.  

We consider the $\mathbb{Z}_2$-graded algebras $M_{1,1}(K)$, $UT_{1,1}(K)$, and $UT_{3}(K)_{(0,1,0)}$ with a superinvolution, along with their corresponding super tensor products with the Grassmann algebra $E$, naturally endowed with a $\mathbb{Z}_2$-grading and also with a superinvolution. Here $UT_{3}(K)_{(0,1,0)}$ stands for the $3\times 3$ upper triangular matrix algebra $UT_3(K)$  with the $\mathbb{Z}_2$-grading induced by the tuple $(0,1,0)$, i.e., the $\mathbb{Z}_2$-grading is given by $ UT_3(K)_0 = \begin{pmatrix} 
K & 0 & K \\
0 & K & 0 \\
0 & 0 & K
\end{pmatrix} $ and $ UT_3(K)_1 = \begin{pmatrix} 
0 & K & 0 \\
0 & 0 & K \\
0 & 0 & 0
\end{pmatrix} $. 
We regard the resulting algebras as endowed with a graded involution and describe its graded $*$-polynomial identities and the corresponding cocharacters.

\section{Preliminaries}

Let $K$ be a fixed field of characteristic zero and $G$ a group. All algebras we consider will be associative and over $K$. 
The algebra $A$ is a $G$-graded algebra if $ A = \oplus_{g\in G} A_g $ is    a direct sum of the vector subspaces  $A_g$, satisfying  the relations
 $A_gA_h \subseteq A_{gh}$   for every $ g$, $h\in G$.  The subspaces $A_g$ are called the {\sl homogeneous components of $A$}, and the elements of each $A_g$ are called the {\sl homogeneous elements of $A$} of homogeneous degree $g$.

Consider  $X = \{x_1, x_2, \dots\}$ and let $K\langle X, * \rangle = K\langle x, x^* \mid x \in X \rangle$  be the free associative algebra with involution generated by $X$ over $K$. 
 Recall that in the cases of algebras with involution over a field of characteristic  different from 2, it is useful to consider the free associative algebra with involution as freely generated by symmetric and skew-symmetric variables, by putting $y_i := x_i + x_i^*$ and $z_i := x_i - x_i^*$, for every $i$. In a similar way, we consider the free associative graded algebra with graded involution freely generated by homogeneous symmetric and homogeneous skew-symmetric variables, as follows.  Let $Y = \{y_{i,g}  \mid i \in \mathbb{N} , g \in G\}$, $Z = \{z_{i,g}  \mid i \in \mathbb{N} , g \in G\}$ be two countable disjoint sets of variables. We denote by $\deg_G y_{i,g} = \deg_G z_{i,g} = g$ the $G$-degree of the variables $Y \cup Z$ with respect to the $G$-grading. Then $Y_g = \{y_{i,g} \mid i \in \mathbb{N}\}$, $Z_g = \{z_{i,g} \mid i \in \mathbb{N}\}$ are homogeneous variables of $G$-degree $g \in G$. The homogeneous degree of a monomial is defined as the product, in $G$, of the degrees of its variables, starting from left to right.

We can define $*$-action on the monomials over $Y \cup Z$ by the equalities 
\begin{equation} \label{eq: *FYZ}
    (x_{i_1} \cdots x_{i_n} )^* = x_{i_n}^* \cdots x_{i_1}^*, \,  \text{ where } \, 
    y_{i,g}^* = y_{i,g},   z_{i,g}^* = -z_{i,g},  x_j \in Y \cup Z
\end{equation}
 where the linear extension of this action is an involution on the free associative algebra $K \langle  Y, Z\rangle$ generated by the set $Y \cup  Z$.
The algebra $\mathcal{F} =  K \langle  Y, Z\rangle$ is $G$-graded with the grading $\mathcal{F} = \oplus_{g \in G} \mathcal{F}_g$ defined by \[ \mathcal{F}_g = \Span_F \{x_{i_1} x_{i_2} \cdots  x_{i_n} \mid  \deg_G x_{i_1} \cdots  \deg_G x_{i_1} = g, \, x_j \in  Y \cup  Z\}. \] It is clear that in the case of the group $G$ be abelian,  the involution (\ref{eq: *FYZ}) is graded. The algebra $\mathcal{F}$ is the free associative graded algebra with graded involution and its elements are called {\sl graded $*$-polynomials}. We are interested in the group  $G=\mathbb{Z}_2$, that is, we consider $\mathbb{Z}_2$-graded identities with graded involution. Hence we consider the free graded algebra with graded involution $  K\langle Y,Z \rangle $, where $Y = Y_0 \cup Y_1$ is the set of symmetric variables of degree $0$ and $1$, and $Z= Z_0 \cup Z_1$ is the set of skew-symmetric variables of degree $0$ and $1$. Thus $y_{i,0}$ and $y_{i,1}$ ($z_{i,0}$ and $z_{i,1}$) denote symmetric (skew-symmetric) variables of degree $0$ and degree $1$, respectively.

Given $A$ a $\mathbb{Z}_2$-graded algebra with graded involution and a graded $*$-polynomial 
\[
f(y_{1,0},\ldots, y_{n_1,0},y_{1,1},\ldots, y_{n_2,1},z_{1,0},\ldots, z_{n_3,0},z_{1,1},\ldots, z_{n_4,1})  \in F\langle Y \cup  Z \rangle,
\]
we say that $f$ is a {\sl graded $*$-polynomial identity} for $A$, if for every $u_{1}^{+}$, \dots, $u_{m}^{+} \in A_{0}^{+}$, $u_{1}^{+}$, \dots, $u_{n}^{-} \in A_{1}^{+}$, $v_{1}^{+}$, \dots, $v_{s}^{+} \in A_{0}^{-}$ and  $v_{1}^{-}$, \dots, $v_{t}^{-} \in A_{1}^{-}$, it holds 
\[ f(u_{1}^{+},\dots,u_{m}^{+}, u_{1}^{-},\dots,u_{n}^{-}, v_{1}^{+},\dots,v_{s}^{+},v_{1}^{-},\dots,v_{t}^{-}) = 0. \]
We denote by $\Id_{\mathbb{Z}_2}^{*}(A) = \{ f \in K\langle Y \cup  Z\rangle  \mid f \equiv 0 \text{ on } A \}$ the $T_2^{*}$-ideal of graded $*$-identities of $A$, i.e., $\Id_{\mathbb{Z}_2}^{*}(A)$ is an ideal of $K\langle Y \cup  Z\rangle$, invariant under all $\mathbb{Z}_2$-graded endomorphisms of the free $\mathbb{Z}_2$-graded algebra commuting with the involution $*$. It is  known that in characteristic zero, every graded  $*$-identity is equivalent to a finite system of multilinear graded $*$-identities.  Thus, we can consider  $P_{n_1,n_2,n_3,n_4}$  the space of multilinear polynomials in $n_1$ symmetric variables of degree $0$, $n_2$ symmetric variables of degree $1$,  $n_3$ skew-symmetric variables of degree $0$   and $n_4$ skew-symmetric variables of degree $1$. 

We consider the characters of $\mathbb{Z}_2$-graded algebras with graded involutions.  For $n_1$, $n_2$, $n_3$, $n_4 \geq 0$, consider  the action of the group $S_{n_1}\times S_{n_2} \times S_{n_3}  \times S_{n_4}$ on $P_{n_1,n_2, n_3, n_4}$ given by 
\[ \begin{split}
	(\omega,\sigma, \tau, \rho)f(y_{1,0},\dots, y_{n_1,0},y_{1,1},\dots, y_{n_2,1},z_{1,0},\dots, z_{n_3,0},z_{1,1},\dots, z_{n_4,1}) \\ 
	=	
	f(y_{\omega(1),0},\dots, y_{\omega(n_1),0},y_{\sigma(1),1},\dots, y_{\sigma(n_2),1},z_{\tau(1),0},\dots, z_{\tau(n_3),0},z_{\rho(1),1},\dots, z_{\rho(n_4),1}),
\end{split}  \] 
where $(\omega,\sigma, \tau, \rho) \in S_{n_1} \times S_{n_2} \times S_{n_3} \times S_{n_4}$ and $f \in P_{n_1,n_2, n_3, n_4}$.

Given a $\mathbb{Z}_2$-graded algebra $A$ with a graded involution, denote by $P_{n_1,n_2,n_3,n_4}(A)$ the quotient space \[ P_{n_1,n_2,n_3,n_4}(A) = \frac{ P_{n_1,n_2,n_3,n_4}}{ P_{n_1,n_2,n_3,n_4}\cap \Id_{\mathbb{Z}_2}^*(A)}.\] 
We examine the action of $S_{n_1} \times S_{n_2} \times S_{n_3} \times S_{n_4}$ on $P_{n_1,n_2,n_3,n_4}(A)$.  

Since $T_2^*$-ideals are invariant under permutations of the variables, we obtain that $P_{n_1,n_2,n_3,n_4} \cap  \Id_{\mathbb{Z}_2}^*(A)$ is a left $S_{n_1} \times S_{n_2} \times S_{n_3} \times S_{n_4}$-submodule of $P_{n_1,n_2,n_3,n_4}$. Thus, $P_{n_1,n_2,n_3,n_4}(A)$  has a structure of left $S_{n_1} \times S_{n_2} \times S_{n_3} \times S_{n_4}$-module, and its character $\chi_{n_1,n_2,n_3, n_4}(A)$ is called the {\sl $(n_1,n_2,n_3,n_4)$-th cocharacter of $A$}. 

The $S_{n_1} \times S_{n_2} \times S_{n_3} \times S_{n_4}$-characters are obtained from the outer tensor product of irreducible characters of $S_{n_1}$, $S_{n_2}$, $S_{n_3}$ and $S_{n_4}$, and we have 1-1 correspondence between the irreducible characters of $S_{n_1} \times S_{n_2} \times S_{n_3} \times S_{n_4}$ and the 4-tuples of partitions $(\omega,\sigma, \tau, \rho)$, where $\omega \vdash n_1 $, $\sigma \vdash n_2$, $\tau \vdash n_3$ and $\rho \vdash n_4$. We denote by $\chi_{\omega}\otimes \chi_{\sigma} \otimes \chi_{\tau}\otimes \chi_{\rho}$ the irreducible  $S_{n_1} \times S_{n_2} \times S_{n_3} \times S_{n_4}$-character associated to the 4-tuple of partitions $(\omega,\sigma, \tau, \rho)$. Thus, 
\begin{equation}
	\chi_{n_1,n_2,n_3, n_4}(A) = \sum_{(\omega,\sigma,\tau,\rho)\vdash (n_1,n_2,n_3,n_4)} m_{\omega,\sigma,\tau,\rho} \chi_{\omega}\otimes \chi_{\sigma} \otimes \chi_{\tau}\otimes \chi_{\rho},
\end{equation}
where  $m_{\omega,\sigma,\tau,\rho}$ is the multiplicity of   $\chi_{\omega}\otimes \chi_{\sigma} \otimes \chi_{\tau}\otimes \chi_{\rho}$. 	The multiplicity $m_{\omega,\sigma,\tau,\rho}$ is equal to the number of linearly independent highest weight vectors corresponding to a standard Young Tableau of shape $(\omega,\sigma, \tau, \rho)$, modulo $\Id_{\mathbb{Z}_2}^*(A)$  (see \cite{drensky2000free} for more details).

Consider again  $K\langle Y, Z\rangle = K\langle Y_0, Y_1, Z_0, Z_1\rangle $, the free unitary $\mathbb{Z}_2$-graded algebra with graded involution, freely generated by homogeneous symmetric and homogeneous skew-symmetric variables. Denote by $\Gamma$ the unitary subalgebra of $K\langle Y, Z\rangle$ generated by the elements from $Y_1 \cup Z$ and all non-trivial (long) commutators in the free variables of $K\langle Y, Z\rangle$. The elements of $\Gamma$  are called {\sl $Y_0$-proper polynomials}. We denote by $\Gamma_{n_1,n_2,n_3, n_4}$ the vector spaces 
	\[  \Gamma_{n_1,n_2,n_3, n_4} = \Gamma \cap P_{n_1,n_2,n_3, n_4}. \] 
 Given $A$ a unitary $\mathbb{Z}_2$-graded algebra with graded involution, we consider the spaces 	\[  \Gamma_{n_1,n_2,n_3,n_4}(A) =  \frac{ \Gamma_{n_1,n_2,n_3,n_4}}{ \Gamma_{n_1,n_2,n_3,n_4}\cap \Id_{\mathbb{Z}_2}^*(A)}.  \] 
	Since $K$ is of characteristic zero, $\Id_{\mathbb{Z}_2}^*(A)$ is generated, as a $T_2^*$-ideal, by multilinear $Y_0$-proper polynomials.

\section{Graded $*$-identities of $M_{1,1}(E)$}

We first consider the matrix algebra $M_2(K)$ with the natural $\mathbb{Z}_2$-grading, denoted by $M_{1,1}(K)$. 
By \cite{gomez1998lie}, we have two superinvolutions  on  $ M_{1,1}(K)$ given by 
\[ 
\begin{pmatrix} a & b \\ c & d \end{pmatrix}^{\circ} =  \begin{pmatrix} d & b \\ -c & a \end{pmatrix} \quad \text{and} \quad \begin{pmatrix} a & b \\ c & d \end{pmatrix}^{\bullet} =  \begin{pmatrix} d & -b \\ c & a \end{pmatrix}. 
\]  
Consider the algebra $M_{1,1}(E) = M_{1,1}(K) \hat{\otimes} E = \left\{ \begin{pmatrix}  a & b \\ c & d \end{pmatrix} \mid a,d \in E_0,   b,c \in E_1  \right\}$. One defines on $M_{1,1}(E)$ the graded involutions:  $*$ induced by $(\circ,i_E)$, $*_1$ induced by  $(\circ,-i_E)$, $*_2$ induced by $(\bullet,i_E)$, and $*_3$ induced by $(\bullet,-i_E)$. 
Here $*=*_3$, $*_1 = *_2$, and these involutions are defined as follows. If $a$, $d \in E_0$, $b$, $c\in E_1$, then
\[ 
\begin{pmatrix} a & b \\ c & d \end{pmatrix}^* =\begin{pmatrix}  d & b \\ -c & a \end{pmatrix},  \qquad 
\begin{pmatrix} a & b \\ c & d \end{pmatrix}^{*_1} = \begin{pmatrix}  d & -b \\ c & a \end{pmatrix}.
\]
Note that the linear map $\psi$ defined by $\psi\left( \begin{pmatrix} a & b \\ c & d \end{pmatrix}  \right) = \begin{pmatrix} d & c \\ b & a \end{pmatrix}$ is an
isomorphism of the algebras with involution $(M_{1,1}(E), *)$ and  $(M_{1,1}(E), *_1)$.

Let $R$ be the algebra $M_{1,1}(E) $ endowed with the involution $*$ induced by the pair $(\circ, i_E)$, that is \[ \begin{pmatrix}  a & b \\ c & d \end{pmatrix}^* = \begin{pmatrix}  d & b \\ -c & a \end{pmatrix}. \] 
It is immediate that $R_{0}^{+} = \left\{ \begin{pmatrix}  a & 0 \\ 0 & a \end{pmatrix} \mid a \in E_0   \right\}$, $R_{0}^{-} = \left\{ \begin{pmatrix}  a & 0 \\ 0 & -a \end{pmatrix} \mid a \in E_0   \right\}$,  $R_{1}^{+} = \left\{ \begin{pmatrix}  0 & b \\ 0 & 0 \end{pmatrix} \mid b \in E_1   \right\}$ and $R_{1}^{-} = \left\{ \begin{pmatrix}  0 & 0 \\ c & 0 \end{pmatrix} \mid c \in E_1   \right\}$.

One sees directly that the following polynomials in $ K\langle Y,Z \rangle $ are identities on $M_{1,1}(E)$.

\begin{tabular}{ll}
	1. $y_{i,1}y_{j,1} $, & 2.   $z_{i,1}z_{j,1} $, \\
	3. $ z_{i,1} \circ  z_{j,0} $,  &   4. $ z_{i,0} \circ  y_{j,1}$, \\
	5. $ [z_{i,0},z_{j,0}] $, & 6. $ [y_{i,1},y_{j,0}] $, \\
	7. $ [ z_{i,0},y_{j,1} ][ z_{k,0},y_{l,1} ] $,     &  8. $ [ z_{i,1},z_{j,0} ][ z_{k,1},z_{l,0} ]  $ ,   \\
	9. $ y_{1,1}z_{1,1}y_{2,1} + y_{2,1}z_{1,1}y_{1,1} $, & 10. $ z_{1,1}y_{1,1}z_{2,1} + z_{2,1}y_{1,1}z_{1,1}  $, \\
	11. $[y_{i,0}, x]  $ for all $x \in Y\cup Z $. & 
\end{tabular}

Our goal is to prove the following theorem, which provides a basis for the  graded $*$-identities of $M_{1,1}(E)$. To this end, we make use of $Y_0$-proper polynomials.

\begin{theorem} \label{th:M2(E)}
	Let $ \left(  M_{1,1}(E), * \right) $ be the algebra of $2\times 2$ matrices with entries in Grassmann algebra and with the canonical  $\mathbb{Z}_2$-grading and endowed with the involution $*$ given by  \[ \begin{pmatrix}  a & b \\ c & d \end{pmatrix}^* = \begin{pmatrix}  d & b \\ -c & a \end{pmatrix}. \] 
	Then, its $T^*_2$-ideal of identities is generated, as a $T^*_2$-ideal, by the identities 1-11.
\end{theorem}

From now on, we denote by $\mathcal{H}$ the $T_2^*$-ideal of $K\langle Y,Z \rangle$ generated by the identities 1--11. 
Note that for the last two identities, we can only consider products of $z_{i,1}$ and $y_{j,1}$ in the order $y_{1,1}z_{1,1}y_{2,1}z_{2,1}\cdots y_{1,k}z_{1,k}$ and $ z_{1,1}y_{1,1}z_{2,1}y_{2,1}\cdots z_{1,k}y_{1,k}$. We also note that 
\[ [R_{1}^{-},R_{0}^{-}] \subseteq R_{1}^{-}, \quad  [R_{1}^{+},R_{0}^{-}] \subseteq R_{1}^{+}, \quad [R_{1}^{+},R_{1}^{-}] \subseteq R_{0}^{+}.  
\]
The following three propositions are easy to deduce. The first of them follows directly from Id. (3) and Id. (4) together with the inclusions above, the second can be deduced from the first and again Id. (3) and Id. (4), while the third is a direct consequence of Id. (2) and Id. (1).

\begin{proposition}
	The polynomials   
	\begin{align*}
		&[y_{1,1},z_{i_1,0},z_{i_2,0}, \dots ,z_{i_k,0}] - 2^{k} y_{1,1}z_{i_1,0}z_{i_2,0} \cdots z_{i_k,0}, \\
		&[z_{1,1},z_{i_1,0},z_{i_2,0}, \dots ,z_{i_k,0}] - 2^{k}z_{1,1}z_{i_1,0}z_{i_2,0} \cdots z_{i_k,0} 
	\end{align*} 
	belong to the ideal $\mathcal{H}$ for every $k$.
\end{proposition}

\begin{proposition} \label{pro:y1z0z1}
	The polynomials 
	\[  [y_{1,1},z_{i_1,0},z_{i_2,0}, \dots ,z_{i_k,0},z_{1,1}] = (-1)^{k} [z_{1,1},z_{i_1,0},z_{i_2,0}, \dots ,z_{i_k,0},y_{1,1}].   \] belong to the ideal $\mathcal{H}$, for every $k$.
\end{proposition}

\begin{proposition} \label{pro:gamma11}
	The polynomials 
	\[  y_{1,1}z_{1,1}[y_{2,1},z_{2,1}] -  y_{1,1}z_{1,1}y_{2,1}z_{2,1} \text{ and } z_{1,1}y_{1,1}[y_{2,1},z_{2,1}] + z_{1,1}y_{1,1}z_{2,1}y_{2,1} \] belong to the ideal $\mathcal{H}$.
\end{proposition}

\subsection{$Y_0$-proper polynomials for $M_{1,1}(E)$}

Consider the vector spaces $ \Gamma_{n,l,m,k}(R)$ of all multilinear $Y_0$-proper polynomials on $M_{1,1}(E)$ in the variables 
\[
y_{1,0}, \ldots, y_{n,0}, y_{1,1}, \ldots , y_{l,1}, z_{1,0}, \ldots, z_{m,0}, z_{1,1}, \ldots ,  z_{k,1} 
\]
in the free algebra $K\langle Y \cup Z\rangle$.

By $[y_{i,0},x]=0$ we have  $ \Gamma_{n,l,m,k}(R) = 0$  if  $n>0$. Hence we consider the possibilities for  $ \Gamma_{0,l,m,k}(R)$. It is clear that by the identities in $\mathcal{H}$ we have that: 
\begin{itemize}
	\item $\Gamma_{0,0,0,k}(R) = 0$, if $k > 1$;
	
	\item $\Gamma_{0,0,m,0}(R) = \Span\{ z_{1,0}\cdots z_{m,0}\}$;
	
	\item $\Gamma_{0,0,m,1}(R) = \Span\{ z_{1,0}\cdots z_{m,0}z_{1,1}\}$ and $\Gamma_{0,0,m,k}(R) = 0$,  if $k > 1$;
	
	\item $\Gamma_{0,1,m,0}(R) = \Span\{ z_{i_1,0}\cdots z_{i_{m},0}y_{1,1}\}$ and $\Gamma_{0,l,m,0}(R) = 0$, if $k > 1$.
\end{itemize}

Consider now $\Gamma_{0,l,0,k}$. By Proposition \ref{pro:gamma11} we have that  $\Gamma_{0,l,0,l}(R)$ is spanned by polynomials in the forms  $ z_{1,1}y_{1,1}z_{2,1}y_{2,1} \cdots z_{l,1}y_{l,1}$, $ y_{1,1}z_{1,1}y_{2,1}z_{2,1} \cdots y_{l,1}z_{l,1} $ and $[y_{1,1},z_{1,1}]\cdots [y_{l,1},z_{l,1}]$.  Furthermore,  since  $[y_{1,1}, z_{1,1} ] = y_{1,1}z_{1,1} - z_{1,1}y_{1,1}$ we obtain that  \[ \Gamma_{0,l,0,l}(R) = \Span \{ z_{1,1}y_{1,1}z_{2,1}y_{2,1} \cdots z_{l,1}y_{l,1}, \, y_{1,1}z_{1,1}y_{2,1}z_{2,1} \cdots y_{l,1}z_{l,1} \}. \] 
The remaining cases are $\Gamma_{0,l+1,0,l}(R) $ and $\Gamma_{0,l,0,l+1}(R) $, and in these cases we have that 
\begin{itemize}
	\item $\Gamma_{0,l+1,0,l}(R)$ is spanned  by $ y_{1,1}z_{1,1}y_{2,1}z_{2,1} \cdots y_{l,1}z_{l,1}y_{l+1,1} $
	
	\item $\Gamma_{0,l,0,l+1}(R)$ is spanned by $  z_{1,1}y_{1,1}z_{2,1}y_{2,1} \cdots z_{l,1}y_{l,1}z_{l+1,1} $.
\end{itemize}

We consider first $\Gamma_{0,l,m,k}(R)$.

In each commutator, at most one $y_{i,1}$ and one $z_{j,1}$ can appear. The non zero commutators are of the form $ [y_{1,1},z_{i_1,0},z_{i_2,0}, \dots ,z_{i_k,0}] $,  $[z_{1,1},z_{i_1,0},z_{i_2,0}, \dots ,z_{i_k,0}]$, and  $ [y_{1,1},z_{i_1,0},z_{i_2,0}, \dots ,z_{i_k,0},z_{1,1}] $. Also $ \Gamma_{0,l,m,k}(R) = 0 $ if $|l-k|>1$.

Consider the polynomials \begin{itemize}
	\item $ p_l = y_{i_1,1}z_{j_1,1} y_{i_2,1}z_{j_2,1} \cdots y_{i_l,1}z_{j_l,1}$
	
	\item $ p_l^- = y_{i_1,1}z_{j_1,1} y_{i_2,1}z_{j_2,1} \cdots y_{i_l,1}$
	
	\item $ q_l = z_{j_1,1}y_{i_1,1} z_{j_2,1}y_{i_2,1} \cdots z_{j_l,1}y_{i_l,1}$
	
	\item $ q_l^- =  z_{j_1,1}y_{i_1,1} z_{j_2,1}y_{i_2,1} \cdots z_{j_l,1} $
\end{itemize}

Now we deal with $\Gamma_{0,k,m,k}(R)$.

We can generate $\Gamma_{0,k,m,k}(R)$ by products of $z_{i,0}$ and the polynomials $p_l$, $q_l$,  and commutators $ [y_{j,1},z_{i_1,0},z_{i_2,0}, \dots ,z_{i_k,0},z_{t,1}]$.  Now, for the commutators we have that \[ [y_{1,1},z_{i_1,0},z_{i_2,0}, \dots ,z_{i_k,0},z_{1,1}] = 2^{k} z_{i_1,0}z_{i_2,0}\cdots z_{i_k,0} \left( (-1)^{k} y_{j,1}z_{t,1} - z_{t,1}y_{j,1} \right).  \]

Thus, by identities 5, 7, 10 and 11 we conclude that  $\Gamma_{0,k,m,k}$ is spanned,  modulo $\mathcal{H}$, by the monomials 
\begin{itemize}
	\item $  \prod_{i=1}^{m}z_{i,0} y_{1,1}z_{1,1}\cdots y_{1,k}z_{1,k} $,
	
	\item $  \prod_{i=1}^{m}z_{i,0}  z_{1,1}y_{1,1}\cdots z_{1,k}y_{1,k}  $.
\end{itemize}

Finally we consider $\Gamma_{0,k,m,k+1}(R)$ and $\Gamma_{0,k+1,m,k}(R)$.

We take $p_l^-$ and $q_l^-$ and then extend the analysis from the case $\Gamma_{0,k,m,k}(R)$ to  $\Gamma_{0,k+1,m,k}(R)$ and $\Gamma_{0,k,m,k+1}(R)$, and we  conclude that

\begin{itemize}
	\item $\Gamma_{0,k+1,m,k}(R)$ is generated  by $  \prod_{i=1}^{m}z_{i,0} y_{1,1}z_{1,1}\cdots y_{1,k}z_{1,k}y_{1,k+1}$, 
	
	\item  $\Gamma_{0,k,m,k+1}(R)$  is generated by $ \prod_{i=1}^{m}z_{i,0}  z_{1,1}y_{1,1}\cdots z_{1,k}y_{1,k}z_{1,k+1}$. 
\end{itemize}

Note that modulo $\mathcal{H}$ the elements of the subspaces  $ \Gamma_{n,l,m,k}$ are linearly independent. Therefore, Theorem \ref{th:M2(E)} has been proven.

\subsection{Cocharacters of  $(M_{1,1}(E),*)$}
We study the cocharacters of $(M_{1,1}(E),*)$. Consider the vector spaces $ P_{n_1,n_2,n_3,n_4}(R)$, and let $\langle \lambda \rangle$ be a multipartition of $n$, $\langle \lambda \rangle = ( \lambda(1), \lambda(2), \lambda(3), \lambda(4) )$, where $\lambda(i) \vdash n_i$, $1\leq i \leq 4$. Let us consider the $(n_1,n_2,n_3,n_4)$-th cocharacter of  $ R= \left( M_{1,1}(E), * \right) $,
\begin{equation} \label{eq:m-M_2E}
	\chi_{n_1,n_2,n_3,n_4}(R) = \sum_{\langle \lambda \rangle \vdash (n_1,\dots,n_4)} m_{\langle \lambda \rangle} \chi_{\lambda(1)} \otimes\chi_{\lambda(2)}\otimes \chi_{\lambda(3)} \otimes \chi_{\lambda(4)}.
\end{equation}
\begin{theorem}
	Let \[  \chi_{n_1,n_2,n_3,n_4}(R) = \sum_{\langle \lambda \rangle \vdash (n_1,\dots,n_4)} m_{\langle \lambda \rangle} \chi_{\lambda(1)} \otimes\chi_{\lambda(2)}\otimes \chi_{\lambda(3)} \otimes \chi_{\lambda(4)} \]
	be the $(n_1,n_2,n_3,n_4)$-th cocharacter of  $ R  = \left( M_{1,1}(E), * \right)$. 
	\begin{itemize}
		\item[(i)] If $\langle \lambda \rangle = ((n_1), \emptyset, (n_3), \emptyset)$ with $n_1 + n_3 \geq 1$, then $m_{\langle \lambda \rangle} = 1$;
		
		\item[(ii)] If $\langle \lambda \rangle = ((n_1), (1^{n_2}), (n_3), (1^{n_2}))$ with $n_2 \geq 1$, then $m_{\langle \lambda \rangle} = 2$;
		
		\item[(iii)] If $\langle \lambda \rangle = ((n_1), (1^{n_2 +1}), (n_3), (1^{n_2}))$ with $n_2 \geq 0$, then $m_{\langle \lambda \rangle} = 1$;
		
		\item[(iv)] If $\langle \lambda \rangle = ((n_1), (1^{n_2}), (n_3), (1^{n_2 + 1}))$ with $n_2 \geq 0$, then $m_{\langle \lambda \rangle} = 1$;
	\end{itemize}
	In all remaining cases $m_{\langle \lambda \rangle} = 0$.
\end{theorem}

\begin{proof}
	From the identities in Theorem \ref{th:M2(E)}, $m_{\langle \lambda \rangle} = 0$ if: $h(\lambda(1)) > 1$, $h(\lambda(3)) > 1$ or  $|n_2 - n_4| >1$. 
	Given a multipartition $ \langle \lambda \rangle$ we  consider the corresponding Young diagrams filled in a standard way. 
	
	If $T_{\lambda(1)}$, $T_{\lambda(2)}$, $T_{\lambda(3)}$, and $T_{\lambda(4)}$ are the corresponding tableaux, due to the identities of  Theorem \ref{th:M2(E)} we can fill $T_{\lambda(1)}$ with the integers $1$, $2$, \dots, $n_1$, and $T_{\lambda(3)}$ with $n_1 + 1$, $n_1 + 2$, \dots, $n_1 + n_3$. Also, by the identities $y_{i,1}y_{j,1} = 0$ and  $z_{i,1}z_{j,1} = 0$ we can't write  consecutive integers in the same tableau  $T_{\lambda(2)}$ or $T_{\lambda(4)}$. So, we consider the tableaux  $T_{\lambda(2)}$ and $T_{\lambda(4)}$ with the integers $n_1 + n_3 + 1$, \dots, $n_1 + n_2 + n_3 + n_4$. 	
	For convenience of notation, we  consider  the tableaux  $T_{\lambda(2)}$ and $T_{\lambda(4)}$ with the integers $ 1$, $2$, \dots, $m$. 
	
	Let us consider first $\langle \lambda \rangle = ( (n_1), (1^{n_2}), (n_3),  (1^{n_2}) )$ and $m = n_2 + n_4 = 2n_2$, with $n_2 >0$. For it, considering the standard tableaux, for $T_{\lambda(2)}$ and $T_{\lambda(4)}$ we have only the possibilities: 
	$$
	T_{\lambda(2)}= \begin{tabular}{|c|}
		\hline 1 \\
		\hline 3 \\
		\hline
		\vdots \\
		\hline $m-1$ \\
		\hline
	\end{tabular}\,, \qquad
	T_{\lambda(4)}= \begin{tabular}{|c|}
		\hline 2 \\
		\hline 4 \\
		\hline
		\vdots \\
		\hline $m$ \\
		\hline
	\end{tabular} 
	\quad\text{ or }\quad
	\overline{T}_{\lambda(2)}= \begin{tabular}{|c|}
		\hline 2 \\
		\hline 4 \\
		\hline
		\vdots \\
		\hline $m$ \\
		\hline
	\end{tabular}\,, \qquad
	\overline{T}_{\lambda(4)}= \begin{tabular}{|c|}
		\hline 1 \\
		\hline 3 \\
		\hline
		\vdots \\
		\hline $m-1$ \\
		\hline
	\end{tabular} 
	$$
	The corresponding highest weight vectors are given by 
	$$\omega = \sum_{\sigma,\tau \in S_{p}}(\sign\sigma)(\sign\tau) y_{\sigma(1),1}z_{\tau(1),1} \cdots y_{\sigma(p),1}z_{\tau(p),1}$$ and 
	$$\overline{\omega}= \sum_{\sigma,\tau \in S_{p}}(\sign\sigma)(\sign\tau) z_{\tau(1),1} y_{\sigma(1),1} \cdots z_{\tau(p),1} y_{\sigma(p),1}.$$
	Modulo the identities satisfied by $(M_{1,1}(E),*)$, we can rewrite the above polynomials as 
	$\omega = (p!)^{2} y_{1,1}z_{1,1} \cdots y_{p,1}z_{p,1}$
	and 
	$ \overline{\omega} = (p!)^{2} z_{1,1} y_{1,1} \cdots z_{p,1} y_{p,1}$.

	For  $\langle \lambda \rangle = ( (n_1), (1^{n_2 + 1}), (n_3),  (1^{n_2}) )$ with $n_2 \geq 0$ and $m +1 = 2n_2 + 1$, we have
	\[
	T_{\lambda(2)}= \begin{tabular}{|c|}
		\hline  1 \\
		\hline 3 \\
		\hline
		\vdots \\
		\hline $ m-1 $ \\
		\hline $ m+1 $ \\
		\hline
	\end{tabular}\,, \qquad
	T_{\lambda(4)}= \begin{tabular}{|c|}
		\hline 2 \\
		\hline 4 \\
		\hline
		\vdots \\
		\hline $m$ \\
		\hline
	\end{tabular} \, .
	\]
	Here the corresponding highest weight vector
	\[
 \omega^{+} = \sum_{\sigma,\tau }(\sign\sigma)(\sign\tau) y_{\sigma(1),1}z_{\tau(1),1} \cdots y_{\sigma(p),1}z_{\tau(p),1}y_{\sigma(p+1),1}
 \]
	can be rewritten as $\omega^{+} =(p+1)!p! y_{1,1}z_{1,1} \cdots y_{p,1}z_{p,1}y_{p+1,1}$.

	Finally, for  $\langle \lambda \rangle = ( (n_1), (1^{n_2}), (n_3),  (1^{n_2 +1}) )$ with $n_2 \geq 0$ we have
	\[
	T_{\lambda(2)}= \begin{tabular}{|c|}
		\hline 2 \\
		\hline 4 \\
		\hline
		\vdots \\
		\hline $m$ \\
		\hline
	\end{tabular}\, , \qquad
	T_{\lambda(4)}= \begin{tabular}{|c|}
		\hline 1 \\
		\hline 3 \\
		\hline
		\vdots \\
		\hline $m-1$ \\
		\hline $m+1$ \\
		\hline
	\end{tabular}\, .
\]
	Here the corresponding highest weight vector is
\[
\overline{\omega}^{+} = \sum_{\sigma,\tau}(\sign\sigma)(\sign\tau) z_{\tau(1),1} y_{\sigma(1),1} \cdots z_{\tau(p),1} y_{\sigma(p),1}z_{\tau(p+1),1}. 
\]
	It can be rewritten as $\overline{\omega}^{+} = (p+1)!p!  z_{1,1} y_{1,1} \cdots z_{p,1} y_{p,1}z_{p+1,1}$.

	The highest weight vectors are not polynomial identities of $A$, and  $\omega$, $\overline{\omega}$ are linearly independent. Therefore,  since $\dim  P_{n_1,n_2,n_3,n_2}(R) = 2$ with $n_2>0$, and  $\dim  P_{n_1,n_2 +1,n_3,n_2}(R) =  \dim  P_{n_1,n_2,n_3,n_2 +1}(R) = 1$ with $n_2 \geq 0$, we conclude that the multiplicities $m_{\langle \lambda \rangle}$ are as in the statement.   
\end{proof}

\section{ Graded $*$-identities of $ UT_{1,1}(E) $}

The algebra of upper triangular matrices of size two, $UT_2(K)$,  as graded algebra has, up to an isomorphism, only two possible gradings: the trivial grading and the canonical $\mathbb{Z}_2$-grading given by $UT_2(K) = (UT_2(K))_0 \oplus (UT_2(K))_1$; where $(UT_2(K))_0 = Ke_{1,1} + Ke_{2,2}$ and $(UT_2(K))_1 = Ke_{1,2}$ (see \cite{valenti2002graded}). We denote by $UT_{1,1}(K)$ the algebra $UT_2(K)$ with the canonical $\mathbb{Z}_2$-grading.

The superinvolutions on $UT_{1,1}(K)$ coincide with the graded involutions, and the only graded involutions (up to equivalence) on $UT_{1,1}(K)$ are given by 
\[ \begin{pmatrix}  a & c \\ 0 & b \end{pmatrix}^{\circ} = \begin{pmatrix}  b & c \\ 0 & a \end{pmatrix}, \qquad \begin{pmatrix}  a & c \\ 0 & b \end{pmatrix}^{s} = \begin{pmatrix}  b & -c \\ 0 & a \end{pmatrix} \] (see \cite{ioppolo2018superinvolutions}).

We consider the algebra $UT_{1,1}(E) = \left\{ \begin{pmatrix}  a & c \\ 0 & b \end{pmatrix} \mid a,b \in E_0,  c \in E_1  \right\}$. 
One  can consider on $UT_{1,1}(E) = UT_{1,1}(K) \hat{\otimes} E$  the following graded involutions: $*$ induced by $(\circ,i_E)$, $*_1$ induced by $(s,i_E)$, $*_2$ induced by $(\circ,-i_E)$ and  $*_3$ induced by $(s,-i_E)$, where $*=*_3$, $*_1 = *_2$ and these involutions are defined as follows. If $a,b \in E_0$, $c\in E_1$
\[ \begin{pmatrix}  a & c \\ 0 & b \end{pmatrix}^* = \begin{pmatrix}  b & c \\ 0 & a \end{pmatrix},  \qquad 
\begin{pmatrix}  a & c \\ 0 & b \end{pmatrix}^{*_1} = \begin{pmatrix}  b & -c \\ 0 & a \end{pmatrix}. \]

Let $A = UT_{1,1}(E) $ be endowed with the involution $*$ given by \[ \begin{pmatrix}  a & c \\ 0 & b \end{pmatrix}^* = \begin{pmatrix}  b & c \\ 0 & a \end{pmatrix}. \] Then, 
\[  A_0^+ = \left\{ \begin{pmatrix}  a & 0 \\ 0 & a \end{pmatrix} \mid a  \in E_0   \right\}, \quad  A_0^- = \left\{ \begin{pmatrix}  a & 0 \\ 0 & -a \end{pmatrix} \mid a \in E_0   \right\} ,  \] 
\[  A_1^+ = \left\{ \begin{pmatrix}  0 & c \\ 0 & 0 \end{pmatrix} \mid \in E_1   \right\}, \quad A_1^- = \left\{ \begin{pmatrix}  0 & 0 \\ 0 & 0 \end{pmatrix}    \right\}. \]

We have the following relations in $ K\langle Y,Z \rangle $ modulo $\Id^*(UT_{1,1}(E))$

\begin{tabular}{lll}
	(a) $[y_{i,0},x] = 0$,  & (b) $y_{i,1}y_{j,1}=0$, &  (c) $z_{i,1}=0$, \\
	(d) $[z_{i,0},z_{j,0}]=0$,  &  (e) $z_{i,0} \circ y_{j,1} =0$. & 
\end{tabular}

\subsection{$Y_0$-proper polynomials on $UT_{1,1}(E)$}

Let $ \Gamma_{n,l,m,k}(A)$ be the vector spaces  of all multilinear $Y_0$-proper polynomials on $A = UT_{1,1}(E)$ in the sets of variables $y_{1,0}$, \dots, $y_{n,0}$, $y_{1,1}$, \dots , $y_{l,1}$, $z_{1,0}$, \dots, $z_{m,0}$, $z_{1,1}$, \dots ,  $z_{k,1}$ in the free algebra $K\langle Y \cup Z\rangle$. Here $n$, $l$, $m$, $k \geq 0$. Denote by $\mathcal{I}$ the $T^*_2$-ideal of $K\langle Y \cup Z\rangle$ generated by the corresponding polynomials in the relations (a)--(e). The identity (e) gives us the following proposition.

\begin{proposition}
	The polynomial $[y_{1,1},z_{\sigma(i_1),0},\dots,z_{\sigma(i_k),0}] - (-2)^{k}z_{i_1,0},\dots,z_{i_k,0}y_{1,1} $  belongs to  $\mathcal{I}$.
\end{proposition}
Note that modulo $\mathcal{I}$:
\begin{itemize}
	\item if $k>0$, then by $z_{1,1} = 0$,  $ \Gamma_{n,l,m,k}(A) = \{ 0 \}$,
	
	\item if $n>0$, then by $[y_{1,0},x] =0$,   $ \Gamma_{n,l,m,k}(A) = \{ 0 \}$.
\end{itemize}
Thus, we only need to consider $ \Gamma_{0,l,m,0}(A)$ with $l$, $m \geq 0$.
Since $z_{i,0} \circ y_{j,1} = 0$ and $y_{i,1}y_{j,1} = 0 $ then $ \Gamma_{0,l,m,0}(A) = \{ 0 \}$ if $l>1$, and for the remaining cases we have:
\begin{itemize}
	\item $\Gamma_{0,1,0,0}(A) = \Span\{y_{1,1}\}$, 
	
	\item  $\Gamma_{0,0,m,0}(A) = \Span\{ z_{1,0}z_{2,0}\cdots z_{m,0} \}$, $m\leq 1$,
	
	\item  $\Gamma_{0,1,m,0}(A) = \Span\{ z_{1,0}z_{2,0}\cdots z_{m,0}y_{1,1} \}$, $m\leq 1$.
\end{itemize}
Modulo $\mathcal{I}$ the elements of the subspaces  $ \Gamma_{n,l,m,k}(A)$ are linearly independent, and this proves the following theorem:

\begin{theorem} \label{th:base-UT2E*}
	Let $ \left( UT_{1,1}(E), * \right) $ be the algebra of $2\times 2$ upper-triangular matrices over the Grassmann algebra, with the canonical $\mathbb{Z}_2$-grading and endowed with the involution $*$ given by \[ \begin{pmatrix}  a & c \\ 0 & b \end{pmatrix}^* = \begin{pmatrix}  b & c \\ 0 & a \end{pmatrix}. \]  Then its $T^*_2$-ideal of identities is generated, as a $T^*_2$-ideal, by 
	
	\begin{tabular}{lll}
		$(i)$ $[y_{i,0},x]$, &  $(ii)$ $y_{i,1}y_{j,1}$,   & $(iii)$ $z_{i,1}$,  \\
		$(iv)$ $[z_{i,0},z_{j,0}]$, &   $(v)$ $z_{i,0} \circ y_{j,1}$. & 
	\end{tabular}
\end{theorem}

\subsection{Cocharacters of $(UT_{1,1}(E),*)$}

Let  $\langle \lambda \rangle = ( \lambda(1), \lambda(2), \lambda(3), \lambda(4) )$ be a multipartition of $n$, where $\lambda(i) \vdash n_i$, $1\leq i \leq 4$.

\begin{theorem}
	Let \begin{equation}
	    \label{eq:m-UT2}
	 \chi_{n_1,n_2,n_3,n_4}(A) = \sum_{\langle \lambda \rangle \vdash (n_1,\dots,n_4)} m_{\langle \lambda \rangle} \chi_{\lambda(1)} \otimes\chi_{\lambda(2)}\otimes \chi_{\lambda(3)} \otimes \chi_{\lambda(4)} 
  \end{equation}
	be the $(n_1,n_2,n_3,n_4)$-th cocharacter of  $ A  = \left( UT_{1,1}(E), * \right)$. 
	\begin{itemize}
		\item[(i)] If $\langle \lambda \rangle = ((n_1), \emptyset, (n_3), \emptyset)$ with $n_1 + n_3 \geq 1$, then $m_{\lambda} = 1$;
		
		\item[(i)] If $\langle \lambda \rangle = ((n_1), (1), (n_3), \emptyset)$, then $m_{\lambda} = 1$.
	\end{itemize}
	In all remaining cases $m_{\lambda} = 0$.
\end{theorem}

\begin{proof}
	From the identities in Theorem \ref{th:base-UT2E*}, $m_{\langle \lambda \rangle} = 0$ if  $h(\lambda(1)) > 1$, or  $h(\lambda(3)) > 1$, or $n_4 > 0$, or $n_2 > 1$. In this way  we have to consider the cases $\langle \lambda \rangle = ( (n_1), \emptyset, (n_3), \emptyset )$ and  $\langle \lambda \rangle = ( (n_1), (1), (n_3), \emptyset )$. 	
	We deal with  $\langle \lambda \rangle = ( (n_1, \emptyset, (n_3), \emptyset )$ with $n_1 + n_3 >0$. To this end, consider the tableaux 
\[
	T_{\lambda(1)} = \begin{tabular}{|c|c|c|c|}
		\hline $1$ & $2$ & $\cdots$ & $n_1$ \\
		\hline 
	\end{tabular}\, , \quad
	T_{\lambda(3)} = \begin{tabular}{|c|c|c|c|}
		\hline $n_1 +1$ & $n_1 + 2$ & $\cdots$ & $n_1 + n_3$ \\
		\hline 
	\end{tabular}\,,
 \]
and $T_{\lambda(2)} = T_{\lambda(4)} = \emptyset$.	The corresponding highest weight vector $\omega = (y_{1,0})^{n_1}(z_{1,0})^{n_3}$ is not an identity of $A$. Therefore, $m_{\langle \lambda \rangle} \geq 1$. By the identities of $A$, there is only one linearly independent  highest weight vector corresponding to the multipartition $\langle \lambda \rangle$. Thus, $m_{\langle \lambda \rangle} = 1$.

	For $\langle \lambda \rangle = ( (n_1), (1), (n_3), \emptyset )$ consider the tableaux 
	$$
	T_{\lambda(1)} = \begin{tabular}{|c|c|c|c|}
		\hline $1$ & $2$ & $\cdots$ & $n_1$ \\
		\hline 
	\end{tabular}\,, 
	\quad 
	T_{\lambda(3)} = \begin{tabular}{|c|c|c|c|}
		\hline $n_1 +1$ & $n_1 + 2$ & $\cdots$ & $n_1 + n_3$ \\
		\hline 
	\end{tabular}\,,
	$$
	$$ T_{\lambda(2)} = \begin{tabular}{|c|}
		\hline $n_1 + n_3 +1$  \\
		\hline 
	\end{tabular}\,, 
	\quad   T_{\lambda(4)} = \emptyset. $$
	The corresponding highest weight vector $\omega = (y_{1,0})^{n_1}(z_{1,0})^{n_3}y_{1,1}$ is not an identity of $A$ and from the identities of $A$, it is the  only one.  Therefore, $m_{\langle \lambda \rangle} = 1$.    
\end{proof}

\section{ Graded $*$-identities of $ UT_{(0,1,0)}(E) $}

Consider $UT_3(K)$ the algebra of upper triangular matrices of size three. As in \cite{ioppolo2018superinvolutions} we consider $C= UT_3(K)_{(0,1,0)}$ the algebra $UT_3(K)$ with the $\mathbb{Z}_2$-grading given by \[ C_0 = Ke_{1,1} \oplus Ke_{2,2} \oplus Ke_{3,3} \oplus Ke_{1,3}, \qquad   C_1 = Ke_{1,2} \oplus Ke_{2,3} . \]  
In \cite{ioppolo2018superinvolutions} it was shown that, up to an  isomorphism,  the only superinvolution on $C$ is given by 
\[ \begin{pmatrix}  a & f & d \\ 0 & b & g  \\ 0 & 0 & c \end{pmatrix}^{\circ} = \begin{pmatrix}  c & g & -d \\ 0 & b & f  \\ 0 & 0 & a \end{pmatrix} . \] 
Then, on $G(UT_{3}) = UT_{3}(E)_{(0,1,0)} \hat{\otimes} E$ we consider the following graded involutions:
\begin{itemize}
	\item $*$ induced by $(\circ,i_E)$,
	\item $*_1$ induced by $(\circ,-i_E)$.
\end{itemize}
Let $B$ be the algebra  $G(UT_{3}) $  with involution $*$. That is, if $a$, $b$, $c$, $d \in E_0$, $f$, $g\in E_1$,  
\[ \begin{pmatrix}  a & f & d \\ 0 & b & g  \\ 0 & 0 & c \end{pmatrix}^{*} = \begin{pmatrix}  c & g & -d \\ 0 & b & f  \\ 0 & 0 & a \end{pmatrix} . \]
Then, 
\[  B_0^+ = \left\{  \begin{pmatrix}  a & 0 & 0 \\ 0 & b & 0  \\ 0 & 0 & a \end{pmatrix}  \mid a,b  \in E_0   \right\}, \quad B_0^- = \left\{ \begin{pmatrix}  d & 0 & c \\ 0 & 0 & 0  \\ 0 & 0 & -d \end{pmatrix} \mid c,d \in E_0   \right\}  \] 
\[  B_1^+ = \left\{  \begin{pmatrix}  0 & g & 0 \\ 0 & 0 & g  \\ 0 & 0 & 0 \end{pmatrix} \mid g \in E_1   \right\}, \quad B_1^- = \left\{  \begin{pmatrix}  0 & f & 0 \\ 0 & 0 & -f  \\ 0 & 0 & 0 \end{pmatrix} \mid f \in E_1   \right\} .  \]

The following polynomials in $ K\langle Y,Z \rangle $ are identities of $B$. 

\begin{tabular}{ll}
	$(i)$ $[y_{i,0},y_{j,0}] $, & $(ii)$ $[y_{i,0},z_{j,0}] $  \\
	$(iii)$ $y_{i,1} \circ y_{j,1}$, & $(iv)$ $[y_{i,1},z_{j,1}] $, \\
	$(v)$ $z_{i,1} \circ z_{j,1} $, &  $(vi)$ $y_{i,1}y_{j,0}y_{k,1} + y_{k,1}y_{j,0}y_{i,1} $, \\
	$(vii)$ $z_{i,1}y_{j,0}z_{k,1} + z_{k,1}y_{j,0}z_{i,1} $, &  $(viii)$ $y_{i,1}y_{j,0}z_{k,1} - z_{k,1}y_{j,0}y_{i,1} $, \\
	$(ix)$ $z_{i,0}z_{j,0}z_{k,0} - z_{k,0}z_{j,0}z_{i,0} $, & $(x)$ $ [z_{i,0},z_{j,0}][z_{k,0},z_{l,0}] $, \\
	$(xi)$ $ x_{k,1}[z_{i,0},z_{j,0}] , \quad x \in \{ y,z\} $, & $(xii)$ $ [z_{i,0},z_{j,0}]x_{k,1} , \quad x \in \{ y,z\} $, \\
	$(xiii)$ $ z_{i,0}x_{k,1}z_{j,0} ,  \quad x \in \{ y,z\}  $, & $(xiv)$ $ x_{i,1} z_{k,0} w_{j,1} , \quad x, \, w \in \{ y,z\} $, \\
	$(xv)$ $ x_{i,1} x_{j,1} x_{k,1} , \quad x_{*,1} \in \{ y_{*,1}, z_{*,1} \}.$ & 
\end{tabular}

\noindent
Let $\mathcal{J}$ be the $T_2^*$-ideal of $ K\langle Y,Z \rangle $  generated by these polynomials. We prove 
\begin{theorem} \label{th:C(E)}
	Let $ B=\left( G(UT_{3}), * \right) $ be the Grassmann envelope of the algebra of $3\times 3$ upper-triangular matrices, with the $\mathbb{Z}_2$-grading induced by the tuple $(0,1,0)$, and endowed with the involution $*$ given by  \[ \begin{pmatrix}  a & f & d \\ 0 & b & g  \\ 0 & 0 & c \end{pmatrix}^{*} = \begin{pmatrix}  c & g & -d \\ 0 & b & f  \\ 0 & 0 & a \end{pmatrix} . \]  
	Then its $T^*_2$-ideal  is generated, as a $T^*_2$-ideal, by the identities $(i) - (xv)$.
\end{theorem} 
\begin{proof}
	Since $B_0\cdot B_1, \, B_1\cdot B_0 \subset B_1$, from the identity $(xv)$ we obtain that if $n_2 + n_4 > 2$   then $P_{n_1,n_2,n_3,n_4}(B) = 0$. Hence, we consider the following three cases:
	\begin{itemize}
		\item[$(1)$] $n_2 = 0$ and $n_4 = 0$. 
		
		Considering the multilinear $Y_0$-proper polynomials $\Gamma_{n_1,0,n_3,0}(B)$, since $ [ B_{0}^{-}, B_{0}^{-} ]   \in B_{0}^{-}$, then $\Gamma_{n_1,0,n_3,0}(B) = 0$ if $n_1 >1$. From the identity $z_{i,0}z_{j,0}z_{k,0} - z_{k,0}z_{j,0}z_{i,0} =0$ we have $ [ z_{i,0},z_{j,0},z_{k,0}] = -2 z_{k,0}[z_{i,0},z_{j,0}]$. Thus $\Gamma_{0,0,n_3,0}(B) = \Span\{ [z_{i_1,0}, z_{i_2,0}, \dots, z_{i_{n_3},0}]\}$.

		By  $ [ [x_1,x_2],[x_3,x_4]] = [x_1,x_2,x_3,x_4] - [x_1, x_2, x_4, x_3]$, $ [z_{i,0},z_{j,0}][z_{k,0},z_{l,0}] = 0$, and the Jacobi Identity,  we conclude that modulo $\mathcal{J}$, \[ \Gamma_{0,0,n_3,0}(B) = \Span\{ [z_{l,0}, z_{1,0}, \dots, \widehat{z_{l,0}}, \dots, z_{n_3,0}] \}, \] where $2 \leq l  \leq n_3$ and $\widehat{z_{l,0}}$ means that the variable $z_{l,0}$ can be omitted.  
		
		Suppose \[ \sum_{l=2}^{n_3} \alpha_l [z_{l,0}, z_{1,0}, \dots, \widehat{z_{l,0}}, \dots, z_{n_3,0}] = 0  \pmod\Id_{\mathbb{Z}_2}^{*}(B), \] and that there exists $\alpha_l \neq 0$, then by the evaluation $z_{l,0} = e_{1,3}$ and $z_{i,0} = e_{1,1} - e_{3,3}$, $i\in \{1,2,\dots,n_3 \} \setminus \{l\}$, we obtain $\alpha_l (-2)^{n_3 - 1}e_{1,3} = 0$, a contradiction. Thus, the commutators   $[z_{l,0}, z_{1,0}, \dots, \widehat{z_{l,0}}, \dots, z_{n_3,0}] $ ($2 \leq l  \leq n_3$) are linearly independent.

		\item[$(2)$] $n_2 = 1$ and $n_4 = 0$ (the case $n_2 = 0$, $n_4 = 1$ is analogous).
		
		By $[y_{i,0},y_{j,0}] = 0$, $[y_{i,0},z_{j,0}] = 0$, $ z_{i,0}y_{1,1}z_{j,0} = 0$ and the identities $z_{j_1,0}\cdots z_{j_{n_3},0} y_{1,1} =  z_{1,0}\cdots z_{{n_3},0} y_{1,1}$ and $ y_{1,1} z_{j_1,0}\cdots z_{j_{n_3},0}=  y_{1,1}z_{1,0}\cdots z_{{n_3},0}$ we obtain that $P_{n_1,1,n_3,0}$ is spanned modulo $\mathcal{J}$ by the monomials  
		\begin{align*}
		    &y_{i_1,0}\cdots y_{i_m,0} z_{1,0}\cdots z_{{n_3},0} y_{1,1}  y_{k_1,0}\cdots y_{k_r,0}, \\ 
      &y_{i_1,0}\cdots y_{i_m,0} y_{1,1} y_{k_1,0}\cdots y_{k_r,0}  z_{1,0}\cdots z_{{n_3},0}, 
      \end{align*}
		where $m+r = n_1$, $i_1 < i_2 < \cdots < i_m$,  $k_1 < k_2 < \cdots < k_r$. 
		
		In order to prove that these monomials are linearly independent we consider the polynomial 
		\[ 
		\begin{aligned} 
			f =  & \sum_{I,K} \alpha_{I,K}y_{i_1,0}\cdots y_{i_m,0} z_{1,0}\cdots z_{{n_3},0} y_{1,1}  y_{k_1,0}\cdots y_{k_r,0} \\ 
			& + \sum_{J,H}\beta_{J,H}  y_{j_1,0}\cdots y_{j_s,0} y_{1,1} y_{h_1,0}\cdots y_{h_t,0}  z_{1,0}\cdots z_{{n_3},0}
		\end{aligned}
		\]
		with $m+r = s + t = n_1$, $ I = \{ i_1, \dots, i_m\}$, $ K = \{ k_1, \dots, k_r\}$, $ J = \{ j_1, \dots, j_s\}$, $ H = \{ h_1, \dots, h_t\}$, and $ i_1 < \cdots < i_m $, $ k_1 < \cdots < k_r$, $ j_1 < \cdots < j_s$, $ h_1 < \cdots < h_t$. 
		
		\noindent
		Suppose $f \in \Id_{\mathbb{Z}_2}^{*}(B)$ and that  there exists $\alpha_{I,K} \neq 0$ or $\beta_{I,K} \neq 0$ for some $I$ and $K$. Now, if we consider $y_{i_l,0} = e_{1,1} + e_{3,3}$ for $l = 1$, \dots, $m$, $y_{k_c,0} = e_{2,2}$ for $c = 1$, \dots,  $r$, $y_{1,1} = e_1(e_{1,2} + e_{2,3})$ and $z_{j,0} = e_{1,1} - e_{3,3}$ for  $j = 1$, \dots, $n_3$, we obtain  $\alpha_{I,K}e_1e_{1,2} \pm \beta_{I,K} e_1e_{2,3} = 0$. Therefore, $\alpha_{I,K} =  \beta_{I,K} = 0$, a contradiction.

		\item[$(3)$] $n_2 = 2$ and $n_4 = 0$ (the cases $n_2 = 0$, $n_4 = 2$ and $n_2 = 1$, $n_4 = 1$  are analogous). 
		
		As $ z_{2,0}y_{1,1}y_{2,1}z_{1,0} = z_{1,0}y_{1,1}y_{2,1}z_{2,0} $ and $ y_{2,0}y_{1,1}y_{2,1}y_{1,0} = y_{1,0}y_{1,1}y_{2,1}y_{2,0} $ (because $B_1^+ \cdot B_1^+ \subset B_0^- $), by the identities above  we obtain  $P_{n_1,2,n_3,0}$ is spanned, modulo $\mathcal{J}$, by  
  \[ 
  y_{i_1,0}\cdots y_{i_m,0} z_{1,0}\cdots z_{{l},0} y_{1,1}  y_{j_1,0}\cdots y_{j_s,0} y_{2,1}  y_{i_{m+1},0}\cdots y_{i_{n_1-s},0} z_{l+1,0}\cdots z_{{n_3},0}, 
  \] 
		with $i_1 < i_2 < \cdots < i_{n_1-s}$,  $j_1 < j_2 < \cdots < j_s$. 
		These monomials are linearly independent modulo $\mathcal{J}$: consider the polynomial 
		\begin{align*}
		f = &\sum_{I,J,l}\alpha_{I,J}^l y_{i_1,0}\cdots y_{i_m,0} z_{1,0}\cdots z_{{l},0} y_{1,1}  y_{j_1,0}\cdots \\
  &y_{j_s,0} y_{2,1}  y_{i_{m+1},0}\cdots y_{i_{n_1-s},0} z_{l+1,0}\cdots z_{{n_3},0},
		\end{align*}
		where $ I = \{ i_1, \dots, i_{n_1 - s}\}$, $ J = \{j_1, \dots, j_s \} $,  $i_1 <   \cdots < i_{n_1-s}$,  $j_1 <   \cdots < j_s$. \\
		Suppose $f \in \Id_{\mathbb{Z}_2}^{*}(B)$ and  there exists $\alpha_{I,J}^l \neq 0$ for some $I$, $J$, $l$. Considering the evaluation $y_{i_n,0} = e_{1,1} + e_{3,3}$ for  $n = 1$, \dots, $n_1 -s$, $y_{j_k}=e_{2,2}$ for $ k = 1$, \dots, $s$, $y_{j,1} = e_j(e_{1,2} + e_{2,3})$ for $j =1$, 2, and  $z_{c,0} = e_{1,1} - e_{3,3}$ for $c = 1$, \dots, $n_3$ we obtain $\alpha_{I,J}^l e_1e_2e_{1,3} = 0$. Thus, $\alpha_{I,J}^l = 0$, a contradiction. 
	\end{itemize}
 
	As a consequence of (1)--(3) we have the proof of Theorem \ref{th:C(E)}.
	\end{proof}

\subsection{Cocharacters of $B$}

Let $\langle \lambda \rangle = ( \lambda(1), \lambda(2), \lambda(3), \lambda(4) ) $ be a multipartition of $n = n_1+n_2+n_3+n_4$, where  $\lambda(i) \vdash n_i$, $1\leq i \leq 4$, and consider the $(n_1,n_2,n_3,n_4)$-th cocharacter of  $ B$,
\begin{equation} \label{eq:m-UT_3E}
	\chi_{n_1,n_2,n_3,n_4}(B) = \sum_{\langle \lambda \rangle \vdash (n_1,\dots,n_4)} m_{\langle \lambda \rangle} \chi_{\lambda(1)} \otimes\chi_{\lambda(2)}\otimes \chi_{\lambda(3)} \otimes \chi_{\lambda(4)}.
\end{equation}
Our next aim is to compute the multiplicities of the cocharacters, starting with the case when we only have variables of degree zero, that is $n_2 + n_4 =0$. 
In order to find the cocharacters of $B$, we consider Young diagrams of shape $\langle \lambda \rangle = (\lambda(1),\lambda(2),\lambda(3),\lambda(4)) \vdash (n_1,n_2,n_3n_4)$ and the corresponding standard tableaux. 
Let $\omega$ be the   highest weight vector associated to $\langle \lambda \rangle = (\lambda(1),\emptyset,\lambda(3),\emptyset)$. By the identities $(i)$, $(ii)$, $(ix)$ and $(x)$ we conclude that $\omega = 0$ in the following situations:
\[ h(\lambda(1)) >1, \quad   \begin{tabular}{|c|c|}
	\hline  &  \\
	\hline &  \\ \hline
\end{tabular}  \subseteq  T_{\lambda(3)}, \quad  \begin{tabular}{|c|}
	\hline  \\ \hline  \\ \hline  \\ \hline  
\end{tabular}   \subseteq  T_{\lambda(3)}. \]
Therefore, we consider the cases   $\langle \lambda \rangle = ((n_1),\emptyset,(n_3),\emptyset)$ and \newline  $\langle \lambda \rangle = ((n_1),\emptyset,(n_3 - 1,1),\emptyset)$. 

\begin{proposition}
	Suppose that either  $\langle \lambda \rangle = ((n_1),\emptyset,(n_3),\emptyset)$ with $n_1 + n_3 > 0$ or  $\langle \lambda \rangle = ((n_1),\emptyset,(n_3 - 1,1),\emptyset)$ for $n_3 > 1$, then $m_{\langle \lambda \rangle} = 1 $  in the cocharacter  (\ref{eq:m-UT_3E}). 
\end{proposition}
\begin{proof}
	Since $P_{n_1,0,n_3,0}(B)$ is spanned by the monomial  $ y_{1,0}\cdots y_{n_1,0}z_{1,0}\cdots z_{n_3,0}$ we have that the only (linearly independent) highest weight vector associated to $\langle \lambda \rangle = ((n_1),\emptyset,(n_3),\emptyset)$ is $\omega = (y_{1,0})^{n_1} (z_{1,0})^{n_3}$ which is not an  identity of $A$. Thus, if $\langle \lambda \rangle = ((n_1),\emptyset,(n_3),\emptyset)$ then $m_{\langle \lambda \rangle} =1$.  
	
	Now we examine the case  $\langle \lambda \rangle = ((n_1),\emptyset,(n_3 - 1,1),\emptyset)$. The tableaux with possibly linearly independent highest weight vectors are 
	$$ T_{\lambda(3)} = \begin{tabular}{|c|c|c|c|c|c|c|}
		\hline $n_1+1$ & $n_1+2$ & $\cdots$ & $n_1+i-1$ &  $n_1 + i +1$ &  $\cdots$  & $n_1 + n_3$ \\
		\hline $ n_1 + i$ & \multicolumn{6}{|c}{}  \\
		\cline { 1 - 1 } 
	\end{tabular}
	$$
	\[ 
	T_{\lambda(1)} = \begin{tabular}{|c|c|c|c|}
		\hline $1$ & $2$ & $\cdots$ & $n_1$ \\
		\hline \end{tabular}\,, \quad  T_{\lambda(2)} =  \emptyset, \quad \quad
	T_{\lambda(4)} = \emptyset.
	\]
	Then, from the identity $(ix)$ for each $i = 2$, \dots, $n_3$, the corresponding highest weight vector has the form
	\[ \begin{split}
		\omega_i & = (y_{1,0})^{n_1} \hat{z}_{1,0} (z_{1,0})^{k} \hat{z}_{2,0}  (z_{1,0})^{l}, \qquad k+l = n_3 -2 \\
		& = \begin{cases}
			0, & k \text{ odd } \\
			(y_{1,0})^{n_1}[z_{1,0},z_{2,0}](z_{1,0})^{n_3-2}, & k \text{ even.}
		\end{cases} 
	\end{split} \]
	Thus,  if $\langle \lambda \rangle = ((n_1),\emptyset,(n_3 - 1,1),\emptyset)$ then $m_{\langle \lambda \rangle} =1$.
\end{proof}

Next, we consider the cases when $n_2 + n_3 = 1$.

\begin{proposition} \label{pro:UT3_g*_m_n2=1}
	If $\langle \lambda \rangle = ((p+q,p),(1),(n_3),\emptyset)$, where $p$, $q \geq 0$ and $n_3 >0$, then $m_{\langle \lambda \rangle} = 2(q+1)$ in the cocharacter (\ref{eq:m-UT_3E}).
\end{proposition}
\begin{proof}
	Let  $\langle \lambda \rangle = ((p+q,p),(1),(n_3),\emptyset)$. We determine the linearly independent highest weight vectors associated to the standard Young tableaux of shape  $\langle \lambda \rangle$ which are not  identities of $B$. Let $T_{\lambda(i)}$ be the tableau associated to $\lambda(i)\vdash n_i$ in $\langle \lambda \rangle$.
	
	A basis for $P_{n_1,1,n_3,0}(B)$ is given  by  
	\[ y_{i_1,0}\cdots y_{i_m,0} z_{1,0}\cdots z_{{n_3},0} y_{1,1}  y_{k_1,0}\cdots y_{k_r,0}, \quad y_{i_1,0}\cdots y_{i_m,0} y_{1,1} y_{k_1,0}\cdots y_{k_r,0}  z_{1,0}\cdots z_{{n_3},0}, 
	\] 
	where $m+r = n_1$, $i_1 < i_2 < \cdots < i_m$,  $k_1 < k_2 < \cdots < k_r$. Given a positive integer $t_1$ and $T_{\lambda(2)} = \begin{tabular}{|c|}
		\hline $t_1$  \\
		\hline 
	\end{tabular}$,
	we fill $T_{\lambda(3)}$ with positive integers, all larger than $t_1$ or all less  than  $t_1$. Also, since $[y_{i,0},y_{j,0}] = 0$, we need $t_1$ to be larger than the integers in the first $p$ positions of the first row of $T_{\lambda(1)}$ and less than the integers in the second row of $T_{\lambda(1)}$, or vice versa. 
	
	These remarks reduce our study to the following two cases.
	
	The first case is
	$$ T_{\lambda(1)} = \begin{tabular}{|c|c|c|c|c|c|c|c|c|}
		\hline $n_3+1$ & $n_3+2$ & $\cdots$ & $n_3+p$ & $\cdots$ & $t_1 - 1$ & $t_1 + 1$ & $\cdots$ & $t_2$ \\
		\hline$t_2 + 1$ & $t_2 + 2$ & $\cdots$ & $ n $ & \multicolumn{5}{|c}{}  \\
		\cline { 1 - 4 } 
	\end{tabular}
	$$
	\[  T_{\lambda(2)} = \begin{tabular}{|c|}
		\hline $t_1$  \\
		\hline \end{tabular}\,, \quad
	T_{\lambda(3)} = \begin{tabular}{|c|c|c|c|}
		\hline $1$ & $2$ & $\cdots$ & $n_3$ \\
		\hline \end{tabular}\,, \quad
	T_{\lambda(4)} = \emptyset,
	\]
	where $n_3 < t_1 < t_2 < n$, and the corresponding highest weight vector is \[ f^{1}_{t_1} = (z_{1,0})^{n_3} \underbrace{ \bar{y}_{1,0}\cdots \tilde{y}_{1,0}}_{p}  (y_{1,0})^{i_1 - p} y_{1,1}   \underbrace{ \bar{y}_{2,0}\cdots \tilde{y}_{2,0}}_{p} (y_{1,0})^{i_2 - p}.   \]
	The second case is 
	\[ T_{\lambda(1)} = \begin{tabular}{|c|c|c|c|c|c|c|c|c|}
		\hline $1$ & $2$ & $\cdots$ & $p$ & $\cdots$ & $t_1 - 1$ & $t_1  + n_3 + 1$ & $\cdots$ & $t_2$ \\
		\hline$t_2 + 1$ & $t_2 + 2$ & $\cdots$ & $ n $ & \multicolumn{5}{|c}{}  \\
		\cline { 1 - 4 } 
	\end{tabular} \]
	\[  T_{\lambda(2)} = \begin{tabular}{|c|}
		\hline $t_1$  \\
		\hline \end{tabular}\,, \quad
	T_{\lambda(3)} = \begin{tabular}{|c|c|c|c|}
		\hline $t_1 +1$ & $t_1 + 2$ & $\cdots$ & $t_1 + n_3$ \\
		\hline \end{tabular}\,, \quad
	T_{\lambda(4)} = \emptyset,
	\]
	where $t_1 < t_2 < n$, and the corresponding highest weight vector is \[ f^{2}_{t_1} =  \underbrace{ \bar{y}_{1,0}\cdots \tilde{y}_{1,0}}_{p}  (y_{1,0})^{i_1 - p} y_{1,1}  (z_{1,0})^{n_3} \underbrace{ \bar{y}_{2,0}\cdots \tilde{y}_{2,0}}_{p} (y_{1,0})^{i_2 - p}. \]
	
	We focus on the first case. For $i_1 = p$, $p+1$, \dots, $p+q$, consider the generic highest weight vector
	\[ f^{1}_{i_1} =  (z_{1,0})^{n_3} \underbrace{ \bar{y}_{1,0}\cdots \tilde{y}_{1,0}}_{p}  (y_{1,0})^{i_1 - p} y_{1,1}   \underbrace{ \bar{y}_{2,0}\cdots \tilde{y}_{2,0}}_{p} (y_{1,0})^{i_2 - p}, \] $i_1 + i_2 = n_1$. The polynomials  $f^{1}_{i_1}$ can be written as \[ f^{1}_{i_1} = \sum_{j=0}^{p} (-1)^{j} \binom{p}{j} (z_{1,0})^{n_3} (y_{1,0})^{i_1 - j} (y_{2,0})^{j} y_{1,1} (y_{1,0})^{i_2 - p + j} (y_{2,0})^{p-j}. \] 
	Considering  the evaluations $y_{1,0} = \alpha_1(e_{1,1} + e_{3,3}) + \beta_1 e_{2,2}$, $y_{2,0} = \alpha_2(e_{1,1} + e_{3,3}) + \beta_2 e_{2,2}$, $z_{1,0} = \zeta(e_{1,1} - e_{3,3}) + \delta_1 e_{1,3}$ and $y_{1,1} = \gamma e_1(e_{1,2} + e_{2,3})$, we obtain
	\[ 
	f^{1}_{i_1}  = \left[ \zeta^{n_3} \gamma \alpha_1^{i_1-p}\beta_1^{i_2-p}e_1 (\alpha_1\beta_2 - \alpha_2\beta_1)^{p} \right]e_{1,2}.\]
	Notice that the coefficient $\zeta^{n_3} \gamma  (\alpha_1\beta_2 - \alpha_2\beta_1)^{p} e_1 \cdot e_{1,2}$ is present in every evaluation of $ f^{1}_{i_1}$, for $i_1 = p$, $p+1$, \dots, $p+q$. Therefore, in order to prove their linear independence, we consider only $ \alpha_1^{i_1-p}\beta_1^{i_2-p}$. 
	
	If there exist scalars $\mu_{i_1}$'s such that $\sum_{i_1 = p}^{p+q} \mu_{i_1}f^{1}_{i_1} = 0$, then for every $\alpha_1$, $\beta_1 \in K$, we have \[ \sum_{i_1 = p}^{p+q} \mu_{i_1} \alpha_1^{i_1-p}\beta_1^{i_2-p} = 0. \] Therefore, $\mu_{i_1}$  vanish for all $i_1$, and   $\{ f^{1}_{i_1} \mid i_1 = p, p+1,\dots,p+q \}$ are linearly independent.
	
	Similarly the set  $\{ f^{2}_{i_1} \mid i_1 = p, p+1,\dots,p+q \}$ is linearly independent. Since $f^{1}_{i_1} = \alpha e_{1,2} $ and $f^{2}_{i_1} = \beta e_{2,3}$, then $m_{\langle \lambda \rangle} \geq 2(q+1)$. We have considered all possibilities for the highest weight vectors being linearly independent, as well as generic evaluations,  and thus we conclude that $m_{\langle \lambda \rangle} = 2(q+1)$.
\end{proof}

Similarly, in the case of one skew-symmetric variable of degree $1$, we have

\begin{proposition}
	If $\langle \lambda \rangle = ((p+q,p),\emptyset,(n_3),(1))$, where $p,q \geq 0$ and $n_3 >0$, then $m_{\langle \lambda \rangle} = 2(q+1)$ in the cocharacter (\ref{eq:m-UT_3E}).
\end{proposition}

In the cases $\langle \lambda \rangle = ((p+q,p), (1),(n_3),\emptyset)$ and $\langle \lambda \rangle = ((p+q,p),\emptyset,(n_3),(1))$, the division of the two families of highest weight vectors depends on the presence of $z_{1,0}$, therefore, we consider separately the case when $n_3=0$.

\begin{proposition}
	If $\langle \lambda \rangle = ((p+q,p), (1),\emptyset,\emptyset)$, where $p$, $q \geq 0$, then $m_{\langle \lambda \rangle} = q+1$ in the cocharacter (\ref{eq:m-UT_3E}).
\end{proposition}
\begin{proof}
	Considering the corresponding tableaux and the identities of $B$, we have that given $T_{\lambda(2)} = \begin{tabular}{|c|}
		\hline $t_1$  \\ \hline \end{tabular}$, then we need $t_1$ to be larger than the integers in the first $p$ positions of the first row of $T_{\lambda(1)}$ and less than the integers in the second row of $T_{\lambda(1)}$, or vice versa. Taking into account standard tableaux and looking for highest weight vectors that are linearly independent, we find that the options are determined by
	\[ T_{\lambda(1)} = \begin{tabular}{|c|c|c|c|c|c|c|c|c|}
		\hline $1$ & $2$ & $\cdots$ & $p$ & $\cdots$ & $t_1 -1$ & $t_1 + 1$ & $\cdots$ & $t_2$ \\
		\hline$t_2 + 1$ & $t_2 + 2$ & $\cdots$ & $ n $ & \multicolumn{5}{|c}{}  \\
		\cline { 1 - 4 } 
	\end{tabular}  \]
	\[  T_{\lambda(2)} = \begin{tabular}{|c|}
		\hline $t_1$  \\
		\hline \end{tabular}\,, \quad
	T_{\lambda(3)} = \emptyset , \quad
	T_{\lambda(4)} = \emptyset,
	\]
	where $t_1 < t_2 < n$, and the corresponding highest weight vector is \[ f_{i_1} =  \underbrace{ \bar{y}_{1,0}\cdots \tilde{y}_{1,0}}_{p} (y_{1,0})^{i_1 - p} y_{1,1} \underbrace{ \bar{y}_{2,0}\cdots \tilde{y}_{2,0}}_{p} (y_{1,0})^{i_2 - p}. \]
	By means of a computation similar to that in the proof of Proposition \ref{pro:UT3_g*_m_n2=1}, we establish that the  highest weight vectors $f_{i_1} $ are linearly independent for $i_1 = p$, $p+1$, \dots, $p+q$. But these are all possibilities, we conclude that  $m_{\langle \lambda \rangle} = q+1$.
\end{proof}

Analogously, we have the case of one skew-symmetric variable of degree $1$.

\begin{proposition}
	If $\langle \lambda \rangle = ((p+q,p),\emptyset,\emptyset,(1))$, where $p$, $q \geq 0$, then $m_{\langle \lambda \rangle} = q+1$ in the cocharacter (\ref{eq:m-UT_3E}).
\end{proposition}

Let us now consider the case where $n_2 + n_4 = 2$. First, we recall that $P_{n_1,2,n_3,0}(A)$ is spanned by polynomials of the form \[ y_{i_1,0}\cdots  y_{i_m,0} z_{1,0}\cdots  z_{l,0} y_{1,1}  y_{j_1,0}\cdots  y_{j_s,0} y_{1,1}  y_{i_{m+1},0}\cdots  y_{i_{n_1-s},0}  z_{l+1,0} \cdots z_{n_3,0}.  \] Let $\langle \lambda \rangle = ((p+q,p),(2),(n_3),\emptyset)$, where $p$, $q$, $n_3 \geq 0$, then the corresponding highest weight vectors that are not zero (considering the identities of $B$) will be of the form 
\[ \underbrace{ \bar{y}_{1,0}\cdots \tilde{y}_{1,0}}_{p} (y_{1,0})^{i_1 - p} (z_{1,0})^{j_1} y_{1,1} \underbrace{ \bar{y}_{2,0}\cdots \tilde{y}_{2,0}}_{p} (y_{1,0})^{i_2 - p} y_{1,1} (y_{1,0})^{i_3} (z_{1,0})^{j_2}.  \]
Consider separately the parts \[ m = y_{1,1} y_{k_1,0}\cdots y_{k_p,0} (y_{1,0})^{i_2 - p} y_{1,1}  \] where $k_i \in \{ 1,2\}$, and a generic evaluation $y_{i,0} = a_i(e_{1,1} + e_{3,3}) + b_ie_{2,2}$, $y_{1,1} = c(e_{1,2}+e_{2,3})$ where $a_i$, $b_i \in E_0$ and $c\in E_1$. Then $m = cbc\cdot e_{1,3} = 0$. Therefore, if   \newline $\langle \lambda \rangle = ((p+q,p),(2),(n_3),\emptyset)$ then $m_{\langle \lambda \rangle} = 0$. Similarly for  $\langle \lambda \rangle = ((p+q,p),\emptyset,(n_3),(2))$.

Next, we consider the case when $\langle \lambda \rangle = ((p+q,p),(1),(n_3),(1))$. 
\begin{proposition}
	If $\langle \lambda \rangle = ((p+q,p),(1),(n_3),(1))$, where $p,q \geq 0$, then $m_{\langle \lambda \rangle} = n_3(q+1)$ in the cocharacter (\ref{eq:m-UT_3E}).
\end{proposition}
\begin{proof}
	Observe that $P_{n_1,1,n_3,1}(A)$ is spanned by polynomials of the form \[ y_{i_1,0}\cdots  y_{i_m,0} z_{1,0}\cdots  z_{l,0} y_{1,1}  y_{j_1,0}\cdots  y_{j_s,0} z_{1,1}  y_{i_{m+1},0}\cdots  y_{i_{n_1-s},0}  z_{l+1,0} \cdots z_{n_3,0}.  \]
	The possible  highest weight vectors associated to standard tableaux  of shape $\langle \lambda \rangle$, which are not  identities of $B$, are described by
	\[  f_{i_1,i_2,j_1} = \underbrace{ \bar{y}_{1,0}\cdots \tilde{y}_{1,0}}_{p} (y_{1,0})^{i_1 - p} (z_{1,0})^{j_1} y_{1,1} \underbrace{ \bar{y}_{2,0}\cdots \tilde{y}_{2,0}}_{p} (y_{1,0})^{i_2 - p} z_{1,1} (y_{1,0})^{i_3} (z_{1,0})^{j_2},  \]
	where $i_1+i_2+i_3 = n_1$, $p \leq i_1,i_2 \leq p+q$ and  $j_1+j_2 = n_3$. 
	
	Again, we rewrite $f_{i_1,i_2,j_1}$ as follows:
	\[  \sum_{k=0}^{p} (-1)^{k} \binom{p}{k}  (y_{1,0})^{i_1 - k} (y_{2,0})^{k} (z_{1,0})^{j_1} y_{1,1} (y_{1,0})^{i_2 - p + k} (y_{2,0})^{p-k} z_{1,1} (y_{1,0})^{i_3}   (z_{1,0})^{j_2}.\]
	Consider the evaluation   $y_{1,0} = \alpha_1(e_{1,1} + e_{3,3}) + \beta_1 e_{2,2}$, $y_{2,0} = \alpha_2(e_{1,1} + e_{3,3}) + \beta_2 e_{2,2}$, $z_{1,0} = \zeta(e_{1,1} - e_{3,3}) + \delta_1 e_{1,3}$, $y_{1,1} = e_1(e_{1,2} + e_{2,3})$, and $z_{1,1} = e_2(e_{1,2} - e_{2,3})$. We obtain 
	\[ f_{i_1,i_2,j_1} = \left[ \zeta^{n_3}(\alpha_1\beta_2 - \alpha_2\beta_1)^p e_1 e_2 \right] \left[ (-1)^{j_2} \alpha_1^{n_1-i_2-p} \beta_1^{i_2-p} \right]e_{1,3}. \]
		Since the term  $  \zeta^{n_3}(\alpha_1\beta_2 - \alpha_2\beta_1)^p e_1 e_2  e_{1,3} $ is common to each $f_{i_1,i_2,j_1}$, we can reduce to considering only the terms $ (-1)^{j_2} \alpha_1^{n_1-i_2-p} \beta_1^{i_2-p}$, which depend on $i_2$ and $j_2$, where $p\leq i_2 \leq p+q$ and $0 \leq j_2 \leq n_3$.

	Suppose that there exist $\mu_{i_2,j_2} \in K$ such that \[  \sum_{i_2,j_2} \mu_{i_2,j_2} (-1)^{j_2} \alpha_1^{n_1-i_2-p} \beta_1^{i_2-p} = 0, \] for all $\alpha_1, \beta_1 \in K$. Then,
	\[  \sum_{i_2 = p}^{p+q} \left( \mu_{i_2,0} - \mu_{i_2,1} + \cdots + (-1)^{n_3}  \mu_{i_2,n_3}  \right)\alpha_1^{n_1-i_2-p} \beta_1^{i_2-p} = 0,\] so for each $i_2$ \[ \mu_{i_2,0} - \mu_{i_2,1} + \cdots + (-1)^{n_3}  \mu_{i_2,n_3} = 0. \]
	We  conclude that fixing $i_2$ we have $n_3$ linearly independent variables $\mu_{i_2,k}$. Therefore, $m_{\langle \lambda \rangle} \geq n_3(q+1)$. As we consider generic evaluations, we get $m_{\langle \lambda \rangle} = n_3(q+1)$.
\end{proof}

Similarly to the case when $n_2 + n_4 = 1$, we have some differences between the cases  $n_3>0$ and $n_3=0$. Next, we consider the case $n_2 = n_4 = 1$ and $n_3=0$.

\begin{proposition}
	If $\langle \lambda \rangle = ((p+q,p),(1),\emptyset,(1))$, where $p$, $q \geq 0$, then $m_{\langle \lambda \rangle} = q+1$ in the cocharacter (\ref{eq:m-UT_3E}).
\end{proposition}
\begin{proof}
	Similarly to the previous proposition, we have that the possible linearly independent  highest weight vectors  that are not polynomial identities of $B$ are represented by 
	\[   \begin{split}
		f_{i_1,i_2,j_1} &  = \underbrace{ \bar{y}_{1,0}\cdots \tilde{y}_{1,0}}_{p} (y_{1,0})^{i_1 - p} y_{1,1} \underbrace{ \bar{y}_{2,0}\cdots \tilde{y}_{2,0}}_{p} (y_{1,0})^{i_2 - p} z_{1,1} (y_{1,0})^{i_3} \\
		& =  \sum_{k=0}^{p} (-1)^{k} \binom{p}{k}  (y_{1,0})^{i_1 - k} (y_{2,0})^{k} y_{1,1} (y_{1,0})^{i_2 - p + k} (y_{2,0})^{p-k} z_{1,1} (y_{1,0})^{i_3}         
	\end{split}  \]
	where $i_1+i_2+i_3 = n_1$ and  $p \leq i_1,i_2 \leq p+q$. 
	Considering   $y_{1,0} = \alpha_1(e_{1,1} + e_{3,3}) + \beta_1 e_{2,2}$, $y_{2,0} = \alpha_2(e_{1,1} + e_{3,3}) + \beta_2 e_{2,2}$, $y_{1,1} = e_1(e_{1,2} + e_{2,3})$ and $z_{1,1} = e_2(e_{1,2} - e_{2,3})$, 
	\[ f_{i_1,i_2,j_1} =  \left[ (\alpha_1\beta_2 - \alpha_2\beta_1)^p e_1 e_2 \right] \left[ \alpha_1^{n_1-i_2-p} \beta_1^{i_2-p} \right]e_{1,3}. \]
	As in the previous propositions, we conclude that  $m_{\langle \lambda \rangle} = q+1$.
\end{proof}

Based on the results presented above, and considering the graded $*$-identities of  $A$, we have the description of the cocharacters of $B$.

\begin{theorem}
	Let 
	\[  \chi_{n_1,n_2,n_3,n_4}(B) = \sum_{\langle \lambda \rangle \vdash (n_1,\dots,n_4)} m_{\langle \lambda \rangle} \chi_{\lambda(1)} \otimes\chi_{\lambda(2)}\otimes \chi_{\lambda(3)} \otimes \chi_{\lambda(4)} \]
	be the $(n_1, n_2, n_3, n_4)$-cocharacter of $B$. Then
	
	\begin{itemize}
		\item[$(i)$]  $m_{\langle \lambda \rangle} = 1 $, if $\langle \lambda \rangle = ((n_1),\emptyset,(n_3),\emptyset)$, where $n_1 + n_3 > 0$ or if  \newline $\langle \lambda \rangle = ((n_1),\emptyset,(n_3-1,1),\emptyset)$ and $n_3 > 1$. 
		
		\item[$(ii)$]  $m_{\langle \lambda \rangle} = q+1  $, if   $\langle \lambda \rangle = ((p+q,p),(1),\emptyset,\emptyset)$ or $\langle \lambda \rangle = ((p+q,p),\emptyset,\emptyset,(1))$, where $n_3 > 0$ and $n_1 = 2p+q$.
		
		\item[$(iii)$]  $m_{\langle \lambda \rangle} = 2(q+1)$, if   $\langle \lambda \rangle = ((p+q,p),(1),(n_3),\emptyset)$ or $\langle \lambda \rangle = ((p+q,p),\emptyset,(n_3),(1))$, where $n_3 > 0$ and $n_1 = 2p+q$.
		
		\item[$(iv)$]  $m_{\langle \lambda \rangle} = n_3(q+1) $, if  $\langle \lambda \rangle = ((p+q,p),(1),(n_3),(1))$, where $n_3 > 0$ and $n_1 = 2p+q$.
		
		\item[$(v)$]  $m_{\langle \lambda \rangle} = q+1$, if $\langle \lambda \rangle = ((p+q,p),(1),\emptyset,(1))$, where  $n_1 = 2p+q$.
	\end{itemize}
	In all remaining cases $m_{\langle \lambda \rangle} = 0$.
\end{theorem}

\section*{Acknowledgements}
I thank my supervisor P. Koshlukov for his suggestions and for the various discussions on the topics in this paper. Thanks are due also to L. Centrone for his advice.

\bibliographystyle{amsplain}
\bibliography{matrices_over_grassman_algebra}

\end{document}